\documentclass{amsart}
\usepackage{latexsym,amssymb,amsthm,amsmath,verbatim}
\usepackage{lscape,pdflscape}
\usepackage{enumerate}
\usepackage{tikz}
\usetikzlibrary{arrows,positioning,chains,matrix,scopes}
\usepackage{cite}
\usepackage{multicol}
\begin{document}

\newtheorem{theorem}{Theorem}[section]
\newtheorem{proposition}[theorem]{Proposition}
\newtheorem{definition}[theorem]{Definition}
\newtheorem{corollary}[theorem]{Corollary}
\newtheorem{lemma}[theorem]{Lemma}
\newtheorem{question}[theorem]{Question}

\theoremstyle{definition}
\newtheorem{remark}{Remark}
\newtheorem{example}{Example}

\newcommand{\pr}[1]{\left\langle #1 \right\rangle}
\newcommand{\mR}{\mathcal{R}}
\newcommand{\RR}{\mathbb{R}}
\newcommand{\QQ}{\mathbb{Q}}
\newcommand{\mA}{\mathcal{A}}
\newcommand{\mE}{\mathcal{E}}
\newcommand{\mV}{\mathcal{V}}
\newcommand{\mF}{\mathcal{F}}
\newcommand{\mU}{\mathcal{U}}
\newcommand{\mN}{\mathcal{N}}
\newcommand{\mG}{\mathcal{G}}
\newcommand{\mP}{\mathcal{P}}
\newcommand{\mT}{\mathcal{T}}
\newcommand{\mB}{\mathcal{B}}
\newcommand{\mL}{\mathcal{L}}
\newcommand{\mK}{\mathcal{K}}
\newcommand{\C}{\mathrm{C}}
\newcommand{\mC}{\mathcal{C}}
\newcommand{\mO}{\mathcal{O}}
\newcommand{\mM}{\mathcal{M}}
\newcommand{\KK}{\mathbb{K}}
\newcommand{\FF}{\mathbb{F}}

\newcommand{\hB}{{\mathcal{B}^+}}

\newcommand{\D}{\mathrm{D}}
\newcommand{\0}{\mathrm{o}}
\newcommand{\OD}{\mathrm{OD}}
\newcommand{\Do}{\D_\mathrm{o}}
\newcommand{\sone}{\mathsf{S}_1}
\newcommand{\gone}{\mathsf{G}_1}
\newcommand{\gfin}{\mathsf{G}_\mathrm{fin}}
\newcommand{\sfin}{\mathsf{S}_\mathrm{fin}}
\newcommand{\Em}{\longrightarrow}
\newcommand{\menos}{{\setminus}}
\newcommand{\w}{{\omega}}
\newcommand{\sel}[1]{\mathsf{S}_{#1}}
\newcommand{\game}[1]{\mathsf{G}_{#1}}
\newcommand{\pl}[1]{\textnormal{\texttt{Player}}\textnormal{\texttt{ #1}}}
\newcommand{\win}[2]{\textnormal{\texttt{#1}}\uparrow{\mathsf{G}}_{\mathrm{#2}}}
\newcommand{\Win}[2]{\textnormal{\texttt{#1}}\uparrow{\mathsf{G}}_{#2}}
\newcommand{\los}[2]{\textnormal{\texttt{#1}}\not{\!\uparrow}\,{\mathsf{G}}_{\mathrm{#2}}}
\newcommand{\Los}[2]{\textnormal{\texttt{#1}}\not{\!\uparrow}\,{\mathsf{G}}_{{#2}}}
\newcommand{\Id}{\underline{\textnormal{Id}}}

\title{Bornologies and filters in selection principles on function spaces}
\author[L. F. Aurichi]{Leandro F. Aurichi$^1$}
\thanks{$^1$ Supported by FAPESP (2017/09252-3)}
\address{Instituto de Ci\^encias Matem\'aticas e de Computa\c c\~ao,
Universidade de S\~ao Paulo, Caixa Postal 668,
S\~ao Carlos, SP, 13560-970, Brazil}
\email{aurichi@icmc.usp.br}

\author[R. M. Mezabarba]{Renan M. Mezabarba$^2$}
\thanks{$^2$ Supported by CNPq (140427/2016-3)}
\address{Instituto de Ci\^encias Matem\'aticas e de Computa\c c\~ao,
Universidade de S\~ao Paulo, Caixa Postal 668,
S\~ao Carlos, SP, 13560-970, Brazil}
\email{rmmeza@icmc.usp.br}



\begin{abstract}
We extend known results of selection principles in $C_p$-theory to the context of spaces of the form $C_{\mB}(X)$, where $\mB$ is a bornology on $X$. Particularly, by using the filter approach of Jordan to $C_p$-theory, we show that $\gamma$-productive spaces are productive with a larger class of $\gamma$-like spaces.
 
\end{abstract}
\maketitle
\section{Introduction}

The framework of selection principles, introduced by Scheepers in \cite{Scheep1996}, provides a uniform manner to deal with diagonalization processes that appears in several mathematical contexts since the 1920's. Detailed surveys on this subject are provided in \cite{tsabanextravaganza,sakaischeepers}. Here we present a brief introduction, in order to fix notations.

Given an infinite set $S$, let $\mathcal{A}$ and $\mathcal{C}$ be families of nonempty subsets of $S$. We consider the following classic selection principles:
\begin{itemize}
\item $\sone(\mathcal{A,C})$: for each sequence $(A_n:n\in\omega)$ of elements of $\mathcal{A}$ there is a sequence $(C_n:n\in\omega)$ such that $C_n\in A_n$ for all $n$ and $\{C_n:n\in\omega\}\in \mathcal{C}$;
\item $\sfin(\mathcal{A,C})$: for each sequence $(A_n:n\in\omega)$ of elements of $\mathcal{A}$ there is a sequence $(C_n:n\in\omega)$ such that $C_n\in[A_n]^{<\omega}$ for all $n$ and $\bigcup_{n\in\omega}C_n\in\mathcal{C}$.
\end{itemize}

There are natural infinite games of perfect information associated with these selection principles. In the same setting of the above paragraph, a play of the game $\gone(\mathcal{A,C})$ is defined as follows: for every inning $n<\omega$, \pl{I} chooses an element $A_n\in \mathcal{A}$, and then \pl{II} picks a $C_n\in A_n$; \pl{II} wins the play if $\{C_n:n\in \omega\}\in\mathcal{C}$. The game $\gfin(\mathcal{A,C})$ is defined in a similar way.

For $\textnormal{\texttt{J}}\in\{\textnormal{\texttt{I,II}}\}$, we denote the sentence ``\pl{J} has a winning strategy in the game $\mathsf{G}$'' by $\win{J}{}$, while its negation is denoted by $\texttt{J}\,{\not{\!\uparrow}}\,\,\mathsf{G}$. The interest about these games lies on finding winning strategies for some of the players and, in the topological context, asking how the topological properties of a space determine these strategies for particular instances of families $\mA$ and $\mC$. 

In the next diagram, the straight arrows summarize the general implications between these principles. \footnotesize
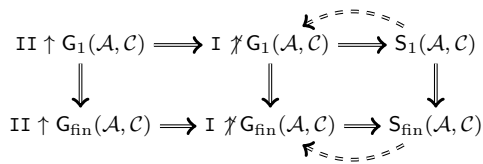
\begin{equation}\begin{tikzpicture}
\matrix(m)[matrix of nodes, row sep=6mm, column sep=5mm,
  jump/.style={text width=10mm,anchor=center},
  txt/.style={anchor=center},]
{
 $\win{II}{1}(\mathcal{A,C})$&$\los{I}{1}(\mathcal{A,C})$&$\sone(\mathcal{A,C})$\\
 
 $\win{II}{fin}(\mathcal{A,C})$&$\los{I}{fin}(\mathcal{A,C})$&$\sfin(\mathcal{A,C})$\\
};
\draw[double, ->] (m-1-1) -- (m-1-2);
\draw[double, ->] (m-1-2) -- (m-1-3);

\draw[double, ->] (m-2-1) -- (m-2-2);
\draw[double, ->] (m-2-2) -- (m-2-3);

\path[dashed, ->] (m-1-3) edge [double,bend right=30] (m-1-2);
\path[dashed, ->] (m-2-3) edge [double,bend left=30] (m-2-2);

\draw[double, ->] (m-1-1) -- (m-2-1);
\draw[double, ->] (m-1-2) -- (m-2-2);
\draw[double, ->] (m-1-3) -- (m-2-3);
\end{tikzpicture}\end{equation}\normalsize

The dashed arrows above mark implications that are not necessarily true in general. An important situation, in which these converses hold, occurs when one takes $\mathcal{A}=\mathcal{C}=\mathcal{O}(X)$, where $\mO(X)$ denotes the family of all open coverings of a topological space $X$.

\begin{theorem}[Hurewicz, 1925]\label{Hur} For a topological space $X$, $\sfin(\mO(X),\mO(X))$ is equivalent to $\los{I}{fin}(\mO(X),\mO(X))$.\end{theorem}
\begin{theorem}[Pawlikowski, 1994]\label{Paw} For a topological space $X$, $\sone(\mO(X),\mO(X))$ is equivalent to $\los{I}{1}(\mO(X),\mO(X))$.\end{theorem}

We introduce in Section~\ref{newsel} a variation of the principles defined above, allowing us to treat simultaneously of several selection principles, based on the ideas presented in \cite{garcia95, ABD}. We shall use these principles in connection with function spaces.

In this context, many dualities are known between selective local properties of $C_p(X)$ and selective covering properties of $X$, where $X$ is a Tychonoff space and $C_p(X)$ denotes the space of the continuous real functions on $X$ with the topology of the pointwise convergence.

Particularly, we are interested in the dualities summarized in the next diagram, where $\Omega$ stands for the collection of $\omega$-coverings of $X$ -- those open coverings $\mU$ such that each finite subset $F$ of $X$ is contained in some element of $\mU$ --, $\Omega_{\0}$ denotes the family $\{A\subset C_p(X):\0\in\overline{A}\}$ and $\0$ is the constant zero function.

\small
\[\begin{tikzpicture}
\matrix(m)[matrix of nodes, row sep=7mm, column sep=8mm,
  jump/.style={text width=10mm,anchor=center},
  txt/.style={anchor=center},]
{
 $\win{II}{1}({\Omega,\Omega})$&$\win{II}{1}(\Omega_{\0},\Omega_{\0})$&$\win{II}{fin}({\Omega,\Omega})$&$\win{II}{fin}(\Omega_{\0},\Omega_{\0})$\\
 $\los{I}{1}(\Omega,\Omega)$&$\los{I}{1}(\Omega_{\0},\Omega_{\0})$&$\los{I}{fin}(\Omega,\Omega)$&$\los{I}{fin}(\Omega_{\0},\Omega_{\0})$\\
$\sone(\Omega,\Omega)$&$\sone(\Omega_{\0},\Omega_{\0})$&$\sfin(\Omega,\Omega)$&$\sfin(\Omega_{\0},\Omega_{\0})$\\
};
\draw[double, <->] (m-1-1) to node[auto]{\cite{Scheep2015}} (m-1-2);
\draw[double, <->] (m-2-1) to node[auto]{\cite{Scheep1997}} (m-2-2);
\draw[double, <->] (m-3-1) to node[auto]{\cite{Sakai1988}} (m-3-2);

\draw[double, ->] (m-1-1) -- (m-2-1);
\draw[double, ->] (m-1-2) -- (m-2-2);

\draw[double, <->] (m-2-1) to node[auto]{\cite{Scheep1997}} (m-3-1);
\draw[double, <->] (m-2-2) -- (m-3-2);

\draw[double, <->] (m-1-3) to node[auto]{\cite{Scheep2015}} (m-1-4);
\draw[double, <->] (m-2-3) to node[auto]{\cite{Scheep1997}} (m-2-4);
\draw[double, <->] (m-3-3) to node[auto]{\cite{Arhanbook} and \cite{Scheep1996b}} (m-3-4);

\draw[double, ->] (m-1-3) -- (m-2-3);
\draw[double, ->] (m-1-4) -- (m-2-4);

\draw[double, <->] (m-2-3) to node[auto]{\cite{Scheep1997}} (m-3-3);
\draw[double, <->] (m-2-4) -- (m-3-4);

\draw[double, ->] (m-1-2) -- (m-1-3);
\draw[double, ->] (m-2-2) -- (m-2-3);
\draw[double, ->] (m-3-2) -- (m-3-3);

\end{tikzpicture}
\]
\normalsize

The work of Caserta et al.~\cite{caserta2012}, generalizing the bottom horizontal equivalences of the diagram above, motivated us to investigate the other equivalences in a broader context, for spaces of the form $C_\mB(X)$, where $\mB$ is a bornology with a compact base. In~\cite{AurMez} we presented generalizations for the horizontal equivalences, but at that time we were not able to solve the vertical ones, originally proved by Scheepers~\cite{Scheep1997} in the context of $C_p$-theory.

In Section~\ref{born} we settle these remaining equivalences, by using the upper semi-finite topology \cite{oldguyMichael} on the family $\mB$, and we analyze its consequences accordingly to the framework presented in Section~\ref{newsel}.
Back to the $C_p$-theory context, Jordan~\cite{jordan2007} obtains general dualities by using filters, and Sections~\ref{filters} and~\ref{productive} are dedicated to extend his results for spaces of the form $C_{\mathcal{B}}(X)$. Particularly, in the last section we show that the class of $\gamma$-productive spaces is \emph{productive} with a class (formally) larger than the class of $\gamma$-spaces -- both definitions are recalled there.

\section{In between $\sone$/$\gone$ and $\sfin$/$\gfin$}\label{newsel}

In this section we fix an infinite set $S$ and families $\mathcal{A}$ and $\mathcal{C}$ of subsets of $S$. We shall denote by $[2,\aleph_0]$ (resp. $[2,\aleph_0)$) the set of all cardinals $\alpha$ such that $2\leq\alpha\leq\aleph_0$ (resp. $2\leq \alpha<\aleph_0)$ and for $n\geq 1$, let $\underline{n}\colon \omega\to[2,\aleph_0)$ be the constant function given by $m\mapsto n+1$ for all $m\in\omega$.

For a function $\varphi\colon\omega\to[2,\aleph_0]$, we consider the following selection principle:
\begin{itemize}
\item $\sel{\varphi}(\mathcal{A,C})$: for each sequence $(A_n:n\in\omega)$ of elements of $\mathcal{A}$ there is a sequence $(C_n:n\in\omega)$ such that $C_n\in[A_n]^{<\varphi(n)}$ for all $n$ and $\bigcup_{n\in\omega}C_n\in\mathcal{C}$.
\end{itemize}

Note that for $\varphi\equiv \aleph_0$, one gets the definition of the selection principle $\sfin$. On the other hand, $\sone(\mathcal{A,C})\Rightarrow\sel{\underline{1}}(\mathcal{A,C})$ and it is formally stronger. However, since $\sel{\underline{1}}(\mathcal{A,C})=\sone(\mathcal{A,C})$ holds for all pairs $(\mathcal{A,C})$ considered along this work, we shall not worry about this, and for simplicity we assume this equality as an additional hypothesis for the general case.
Hence, for each $n\geq 1$ it makes sense to denote the selection principle $\sel{\underline{n}}(\mathcal{A,C})$ as $\sel{n}(\mathcal{A,C})$.

The original prototype of the above selection principle was defined in \cite{garcia95}, where the authors concerned about variations of tightness by taking \[\mathcal{A}=\mathcal{C}=\Omega_x:=\{A\subset X:x\in\overline{A}\}.\] 
In~\cite{ABD}, the natural adaptation of the principle $\sel{\varphi}$ to the context of games was analyzed  for the same pair $(\Omega_x,\Omega_x)$. This motivates our next definition.

For $\mA,\mC$ and $\varphi$ as before, let $\game{\varphi}(\mathcal{A,C})$ be the infinite game of perfect information between $\pl{I}$ and $\pl{II}$, defined as follows:
\begin{itemize}
\item for every inning $n<\omega$, \pl{I} chooses an element $A_n\in \mathcal{A}$, and then \pl{II} picks a $C_n\in [A_n]^{<\varphi(n)}$;
\item \pl{II} wins if $\bigcup_{n\in\omega}C_n\in\mathcal{C}$.
\end{itemize}

Again, the constant function $\varphi\equiv \omega$ yields the game $\gfin(\mathcal{A,C})$, and since we assume $\gone(\mathcal{A,C})=\game{\underline{1}}(\mathcal{A,C})$, we may denote the game $\game{\underline{n}}(\mathcal{A,C})$ as $\game{n}(\mathcal{A,C})$.

The general relationship between these principles is stated in the following
\begin{proposition}
Let $\varphi$ and $\psi$ be functions of the form $\omega\to[2,\aleph_0]$ such that $\psi\leq \varphi$. Then
\[\label{diagramageral}
\begin{tikzpicture}
\matrix(m)[matrix of nodes, row sep=7mm, column sep=6mm,
  jump/.style={text width=10mm,anchor=center},
  txt/.style={anchor=center},]
{
 $\Win{II}{\psi}(\mathcal{A,C})$&$\Los{I}{\psi}(\mathcal{A,C})$&$\sel{\psi}(\mathcal{A,C})$\\
 
 $\Win{II}{\varphi}(\mathcal{A,C})$&$\Los{I}{\varphi}(\mathcal{A,C})$&$\sel{\varphi}(\mathcal{A,C})$\\
};
\draw[double, ->] (m-1-1) -- (m-1-2);
\draw[double, ->] (m-1-2) -- (m-1-3);

\draw[double, ->] (m-2-1) -- (m-2-2);
\draw[double, ->] (m-2-2) -- (m-2-3);

\draw[double, ->] (m-1-1) -- (m-2-1);
\draw[double, ->] (m-1-2) -- (m-2-2);
\draw[double, ->] (m-1-3) -- (m-2-3);
\end{tikzpicture}
\]
\end{proposition}

Particularly, note that for $\psi=\underline{1}$ and $\varphi\equiv\aleph_0$, the above diagram yields the first one presented in Introduction.
Also, for $\mA=\mC=\mO(X)$, one has the following

\begin{theorem}
[García-Ferreira and Tamariz-Mascarúa~\cite{garcia95}]\label{minicurso} $\sel{f}(\mO(X),\mO(X))$ and $\sone(\mO(X),\mO(X))$ are equivalent for any space $X$ and any function $f\colon\omega\to[2,\aleph_0)$.
\end{theorem}

Thus, in this context, all of the following statements are equivalent,
\begin{enumerate}[(i)]\begin{multicols}{2}
\item$\los{I}{1}(\mathcal{O,O})$
\item $\Los{I}{f}(\mathcal{O,O})$
\columnbreak

\item $\sel{f}(\mathcal{O,O})$
\item $\sone(\mathcal{O,O})$,\end{multicols}
\end{enumerate}
because $\sone(\mathcal{O,O})\Rightarrow \los{I}{1}(\mathcal{O,O})$ by Theorem \ref{Paw}.
\begin{remark}
The natural question then is whether the games $\game{f}(\mO(X),\mO(X))$ and $\gone(\mO(X),\mO(X))$ are equivalent or not. Nathaniel Hiers, in a joint work with Logan Crone, Lior Fishman, and Stephen Jackson, recently\footnote{At the Conference Frontiers of Selection Principles, that took place on Warsaw during the two last weeks of August, 2017.} presented an affirmative answer concerning the game $\game{2}$, for any Hausdorff space $X$.

Although their solution can possibly be extended for any function $f\colon\omega\to[2,\aleph_0)$, we mention that in this general case, an affirmative answer can also be obtained when $X$ is a T$_1$ second countable space or a Hausdorff space with G$_\delta$-points.
\end{remark}

However, the above equivalences do not hold in the tightness context. Denoting as $\Id\colon\omega\to[2,\aleph_0)$ the function given by $\Id(n)=n+2$ for all $n\in\omega$, one has the following
\begin{theorem}
[Garc\'ia-Ferreira and Tamariz-Mascar\'ua, \cite{garcia95}]\label{seleq} Let $Y$ be a topological space, $y\in Y$ and let $f\colon\omega\to[2,\aleph_0)$ be a function.
\begin{enumerate}
\item If $f$ is bounded, then $\sel{f}(\Omega_y,\Omega_y)$ is equivalent to $\sone(\Omega_y,\Omega_y)$.
\item If $f$ is unbounded, then $\sel{f}(\Omega_y,\Omega_y)$ is equivalent to $\sel{\Id}(\Omega_y,\Omega_y)$.
\end{enumerate}
\end{theorem}

Examples 3.7 and 3.8 in \cite{garcia95} show that in general the implications
\[\sone(\Omega_y,\Omega_y)\stackrel{\#}{\Longrightarrow}\sel{\Id}(\Omega_y,\Omega_y)\Rightarrow\sfin(\Omega_y,\Omega_y)\]
are not reversible. Still, the authors also show that for spaces of the form $C_p(X)$, where $X$ is a Tychonoff space, the converse of $(\#$) holds. We prove this in the next section in a more general context.
We finish this section with the counterpart of Theorem \ref{seleq} for games, that will be useful later.

\begin{theorem}
[Aurichi, Bella and Dias~\cite{ABD}]\label{abdt} Let $Y$ be a topological space, $y\in Y$ and let $f\colon\omega\to[2,\aleph_0)$ be a function.
\begin{enumerate}
\item If $f$ is bounded, then the games $\game{f}(\Omega_y,\Omega_y)$ and $\game{{k-1}}(\Omega_y,\Omega_y)$ are equivalent, where $k=\limsup_{n\in\omega} f(n)$.
\item If $f$ is unbounded, then $\game{f}(\Omega_y,\Omega_y)$ and $\game{\Id}(\Omega_y,\Omega_y)$ are equivalent.
\end{enumerate}
\end{theorem}

\section{Bornologies as hyperspaces}\label{born}
We recall the basic definitions from \cite{AurMez}. A {bornology} $\mB$ on a topological space $X$ is an ideal of subsets of $X$ that covers the space. A subset $\mathcal{B}'$ of $\mB$ is called a compact base for the bornology $\mB$ if $\mB'$ is cofinal in $\mathcal{B}$ with respect to inclusion and all its elements are compact subspaces of $X$.

For a topological space $X$ and a bornology $\mB$ on $X$, we call the {topology of the uniform convergence on} $\mB$, denoted by $\mathcal{T}_{\mB}$, as the topology on $C(X)$ having as a neighborhood base at each $f\in C(X)$ the sets of the form \[\langle B,\varepsilon\rangle[f]:=\{g\in C(X):\forall x\in B(|f(x)-g(x)|<\varepsilon)\},\]
for $B\in\mB$ and $\varepsilon>0$. By $C_{\mB}(X)$ we mean the space $(C(X),\mathcal{T}_{\mB})$.

It can be showed that $\mT_{\mB}$ is obtained from a separating uniformity over $C(X)$, from which it follows that $C_{\mB}(X)$ is a Tychonoff space (see McCoy and Ntantu \cite{McCoybook}). It is also worth to mention that $C_{\mathcal{B}}(X)$ is a homogeneous space, so there is no loss of generality in fixing an appropriate point from $C_{\mathcal{B}}(X)$ in order to analyze its closure properties -- in this case, we fix the zero function $\0\colon X\to\RR$.

A collection $\mC$ of open sets of $X$ is a $\mB$-covering for $X$ if for every $B\in\mB$ there is a $C\in\mC$ such that $B\subset C$. Following the notation of Caserta {et al.}\cite{caserta2012}, we denote by $\mathcal{O}_{\mB}$ the collection of all open $\mB$-coverings for $X$. When $\mU\in\mO_{\mB}$ is such that $X\not\in\mU$, the $\mathcal{B}$-covering $\mathcal{U}$ is said to be nontrivial. An important fact about nontrivial $\mB$-coverings is that any cofinite subset of it is also a $\mB$-covering of $X$.

The main examples of bornologies with a compact base on a topological space $X$ are the bornologies $\mathcal{F}=[X]^{<\aleph_0}$ and $\mathsf{K}=\{A\subset X:\exists K\subset X$ compact and $A\subset K\}$ $-$ if $X$ is a Hausdorff space, then $\mathsf{K}=\{A\subset X:\overline{A}$ is compact$\}$. For $\mB=\mathcal{F}$, one has $C_{\mF}(X)=C_p(X)$ and the $\mF$-coverings turns out to be the $\omega$-coverings of $X$. Also, if $X$ is Hausdorff, it follows that $C_{\mathsf{K}}(X)=C_{k}(X)$, where $C_k(X)$ denotes the set $C(X)$ with the compact-open topology. One readily sees that $\mO_{\mathsf{K}}=\mK$, where $\mK$ denotes the set of the so called $K$-coverings of $X$.

Now, recall we want to generalize the following theorem for $\mB$-coverings.

\begin{theorem}
[Scheepers~\cite{Scheep1997}]\label{original} Let $X$ be a Tychonoff space.
\begin{enumerate}
\item $\sfin(\Omega,\Omega)$ is equivalent to $\los{I}{fin}(\Omega,\Omega)$.
\item $\sone(\Omega,\Omega)$ is equivalent to $\los{I}{1}(\Omega,\Omega)$.
\end{enumerate}
\end{theorem}

Although the requirement of a compact base is necessary to settle the dualities between local properties of $C_\mB(X)$ and covering properties of $X$, the generalization of the above theorem for $\mB$-covering does not need any requirement on the bornology $\mB$: in fact, it holds for an arbitrary family $\mB$ of subsets of $X$.

\begin{remark}[Scheepers' key idea]\label{key}
It may be enlightening to review Scheepers' original proof. The key idea in his arguments for proving Theorem~\ref{original} consists in finding an appropriate hyperspace $Y=Y(X)$ such that $\sel{\bullet}(\Omega,\Omega)$ in $X$ translates as $\sel{\bullet}(\mO(Y),\mO(Y))$ in $Y$. This is done in such a way that he can carry back and forth strategies and plays from the game $\game{\bullet}(\Omega,\Omega)$ in $X$ to the game $\game{\bullet}(\mO(Y),\mO(Y))$. This allows him to reduce the problem to a scenario where Theorems~\ref{Hur} and~\ref{Paw} are available.\end{remark}

The difficulty in following the above sketch when trying to generalize it to $\mathcal{B}$-coverings consists in finding an appropriate hyperspace $Y(X)$. Scheepers originally used $Y(X)=\sum_{n\in\omega}X^n$, but we were not able to relate this construction to the bornology $[X]^{<\aleph_0}$. The way we found to solve this problem was to consider $Y(X)$ as the bornology itself, with an appropriate topology.

More generally, given a family $\mathcal{B}$ of nonempty subsets of a topological space $X$, we consider the topology on $\mathcal{B}$ whose basic open neighborhoods are sets of the form
\[\langle U\rangle :=\{B\in\mathcal{B}:B\subset U\},\]
for $U\subset X$ open. This type of hyperspace has been studied already in the literature\footnote{We would like to thank Valentin Gutev for pointing this out.}: in~\cite{oldguyMichael}, Michael considers over $\mA(X):=\{A\subset X:A\ne\emptyset\}$ the topology generated by sets of the form
\begin{equation} U^{+}:=\{A\in \mA(X):A\subset U\},\end{equation}
with $U$ ranging over the open sets of $X$, and he calls it as the {upper semi-finite topology} on $\mA(X)$; by restricting this construction to the the family of all nonempty closed subsets of $X$, one obtains the so called {upper Vietoris} topology~\cite{Hola}.

Since the topology on $\mB$ generated by the family $\{\langle U\rangle:U\subset X$ is open$\}$ is the topology of $\mB$ as a subspace of $\mA(X)$, we shall write $\hB$ to denote the family $\mB$ endowed with this topology.

The main problem with the hyperspace $\hB$ concerns its poor separation properties: one readily sees that if there are $A,B\in\mB$ such that $A\subset B$, then they cannot be separated as points of $\hB$, showing that $\hB$ is not T$_1$. However, this lack of separation properties will be harmless in our context.

\begin{lemma}\label{lemma1}
Let $X$ be a topological space and let $\mathcal{B}$ be a family of subsets of $X$.
\begin{enumerate}
\item If $\mathcal{U}$ is a $\mathcal{B}$-covering for $X$, then $\langle \mathcal{U}\rangle:=\{\langle U\rangle:U\in\mathcal{U}\}$ is an open covering for $\hB$.
\item If $\mathcal{W}$ is an open covering for $\hB$ consisting of basic open sets, then the family $\widetilde{\mathcal{W}}:=\{U:\langle U\rangle\in \mathcal{W}\}$ is a $\mathcal{B}$-covering for $X$.
\end{enumerate}
Then, let $\varphi\colon\omega\to[2,\aleph_0]$ be a function.
\begin{enumerate}
\setcounter{enumi}{2}
\item $\sel{\varphi}(\mathcal{O_B,O_B})$ holds in $X$ if and only if $\sel{\varphi}(\mO(\hB),\mO(\hB))$ holds.
\item The games $\game{\varphi}(\mathcal{O_B,O_B})$ in $X$ and $\game{\varphi}(\mO(\hB),\mO(\hB))$ are equivalent.
\end{enumerate}
\end{lemma}
\begin{proof}
The items $(1)$ and $(2)$ follow from the definition of $\hB$. The other items hold because one can replace arbitrary open coverings in $\hB$ with open coverings consisting of basic open sets, what enables one to uses the previous items.
\end{proof}

\begin{theorem}\label{neworiginal}
Let $X$ be a topological space and let $\mathcal{B}$ be a family of subsets of $X$.
\begin{enumerate}
\item If $\sone(\mathcal{O_B,O_B})$ holds, then $\Los{I}{1}(\mathcal{O_B,O_B})$ also holds.
\item If $\sfin(\mathcal{O_B,O_B})$ holds, then $\los{I}{fin}(\mathcal{O_B,O_B})$ also holds.
\end{enumerate}
\end{theorem}
\begin{proof}
Repeat the steps in Remark~\ref{key} with $Y(X)=\hB$.
\end{proof}

\begin{corollary}\label{easy}
Let $X$ be a topological space and let $\mathcal{B}$ be a family of subsets of $X$. For a function $f\colon \omega\to[2,\aleph_0)$, the following are equivalent:
\begin{enumerate}
\item $\los{I}{1}(\mathcal{O_B,O_B})$;
\item $\Los{I}{f}(\mathcal{O_B,O_B})$;
\item $\sel{f}(\mathcal{O_B,O_B})$;
\item $\sone(\mathcal{O_B,O_B})$.
\end{enumerate}
\end{corollary}
\begin{proof}
These equivalences hold for the pair $(\mO(\hB),\mO(\hB))$. Apply Lemma \ref{lemma1} to finish.
\end{proof}

By replacing $\mB$ with $[X]^{<\aleph_0}$ in the above corollary results in a strengthening of Theorem~\ref{original}, while taking $\mB$ as the family of all compact subsets of $X$ yields new results about $K$-coverings. Well, \emph{almost new} results, as we explain below.

\begin{remark}
Following the announcement of this work, Boaz Tsaban brought to our attention that a result similar to Theorem~\ref{neworiginal} also appears in the (thus far, unpublished) MSc thesis of his student Nadav Samet~\cite{Nadav}.

Instead of considering a topology over a family of subsets of $X$, they take a family $\mathbb{P}$ of filters of open sets and observe that sets of the form $O_U:=\{p\in\mathbb{P}:U\in p\}$ define a base for a topology over $\mathbb{P}$ when $U$ ranges over the open sets of $X$. 
\end{remark}

In connection with spaces of the form $C_{\mathcal{B}}(X)$, we first state a generalization of some of our results in \cite{AurMez}.

\begin{proposition}\label{prop1}
Let $X$ be a Tychonoff space and let $\mB$ be a bornology on $X$ with a compact base. Consider a function $\varphi\colon \omega\to[2,\aleph_0]$.
\begin{enumerate}
\item $\sel{\varphi}(\mathcal{O_B,O_B})$ holds in $X$ if and only if $\sel{\varphi}(\Omega_{\0},\Omega_{\0})$ holds in $C_{\mathcal{B}}(X)$.
\item If $\varphi$ is non-decreasing, then the games $\game{\varphi}(\mathcal{O_B,O_B})$ in $X$ and $\game{\varphi}(\Omega_{\0},\Omega_{\0})$ in $C_\mB(X)$ are equivalent.
\end{enumerate}
\end{proposition}

The proof is essentially an adaptation of the arguments presented in \cite{AurMez}, and it follows with appropriate applications of the following lemma, adapted from \cite{caserta2012}.

\begin{lemma}
Let $X$ be a Tychonoff space and let $\mathcal{B}$ be a bornology with a compact base on $X$.
\begin{enumerate}
\item If $\mathcal{U}$ is a collection of open sets of $X$ such that $X\not\in\mathcal{U}$, then $\mathcal{U}\in\mathcal{O_B}$ if and only if
$\mathcal{A}(\mathcal{U})=\{f\in C_{\mathcal{B}}(X):\exists U\in\mathcal{U}\left(f\upharpoonright (X\setminus U)\equiv 1\right)\}\in\Omega_{\0}.$
\item Let $A\subset C_{\mathcal{B}}(X)$, $n\in\omega$ and set $\mU_n(A)=\left\{f^{-1}\left[\left(-\frac{1}{n+1},\frac{1}{n+1}\right)\right]:f\in A\right\}$. If $\0\in\overline{A}$, then $\mathcal{U}_n(A)\in\mathcal{O_B}$.
\item If $(A_n)_{n\in\omega}$ is a sequence of finite subsets of $C_{\mathcal{B}}(X)$ such that $\bigcup_{n\in\omega}\mathcal{U}_n(A_n)$ is a nontrivial $\mathcal{B}$-covering, then $\bigcup_{n\in\omega}A_n\in\Omega_{\0}$.
\item If $(A_n)_{n\in\omega}$ is a sequence of finite subsets of $C_{\mathcal{B}}(X)$ such that $\bigcup_{n\in\omega}A_n\in\Omega_{\0}$ and for each $n\in\omega$ and each $g\in A_n$ there is a proper open set $U_g\subset X$ such that $g\upharpoonright(X\setminus U_g)\equiv 1$, then $\bigcup_{n\in\omega}\{U_g:g\in A_n\}\in \mathcal{O_B}$.
\end{enumerate}
\end{lemma}

\begin{proof}\text{ }

(1)$\quad$If $\mU\in\mO_\mB$ and $\langle B,\varepsilon\rangle[\0]$ is a neighborhood of $\0$, then we obtain a function ${f\in \mA(\mU)\cap \langle B,\varepsilon\rangle[\0]}$, because $\overline{B}$ is compact and $X$ is a Tychonoff space (see \cite[Theorem 3.1.7]{Engelking})\footnote{Particularly, everything still works if $X$ is a normal space and $\mB$ is a bornology with a closed base.}. Conversely, for a $B\in\mB$ we take an $f\in \mA(\mU)\cap \langle B,1\rangle[\0]$, from which we obtain an open set $U\in\mU$ such that $B\subset U$.

(2)$\quad$It follows because for a function $f\in C_{\mB}(X)$, $f\in \langle B,\frac{1}{n+1}\rangle[\0]$ if and only if $B\subset f^{-1}\left[\left(-\frac{1}{n+1},\frac{1}{n+1}\right)\right]$.

(3)$\quad$In addition to the previous observation, we use the fact that if $\mU\in\mO_\mB$ is nontrivial, then $\mU\setminus F\in\mO_\mB$ for any finite subset $F\subset\mU$.

(4)$\quad$For a $B\in\mB$, we take a $g\in\bigcup_{n\in\omega} A_n\cap \langle B,1\rangle[\0]$, from which it follows that $B\subset U_g$, because $g\upharpoonright (X\setminus U_g)\equiv 1$.\qedhere
\end{proof}

Particularly, the monotonicity hypothesis in Proposition \ref{prop1} can be dropped if the function $\varphi$ is of the form $\omega\to[2,\aleph_0)$. 
In fact, this follows from Theorem \ref{abdt} and from its counterpart for $\mathcal{B}$-coverings, which we state below.

\begin{proposition}
 Let $X$ be a topological space with a bornology $\mathcal{B}$, and consider a function $f\colon\omega\to[2,\aleph_0)$.
\begin{enumerate}
\item If $f$ is bounded, then the games $\game{f}(\mathcal{O_B,O_B})$ and $\game{{k}}(\mathcal{O_B,O_B})$ are equivalent, where $k=\limsup_{n\in\omega} f(n)$;
\item If $f$ is unbounded, then $\game{f}(\mathcal{O_B,O_B})$ and $\game{\Id}(\mathcal{O_B,O_B})$ are equivalent.
\end{enumerate}
\end{proposition}
\begin{proof}
In face of Corollary \ref{easy}, we just need to worry about $\pl{II}$. For the first case where $f$ is bounded, note that there exists an $m_0\in\omega$ such that $f(n)\leq k$ for all $n\geq m_0$. Thus, if $\mu$ is a winning strategy for $\pl{II}$ in $\game{f}$, then $\mu$ induces a winning strategy on $\game{k}$ simply by ignoring the $m_0$ first innings -- here we also use the fact that $\mU\setminus F\in\mO_\mB$ whenever $\mU\in\mO_\mB$ is nontrivial and $F\in[\mU]^{<\omega}$. The converse holds because the set $N=\{n\in\omega:f(n)=k\}$ is infinite.

Now, if $f$ and $g$ are unbounded, by symmetry it is enough to show that \[\Win{II}{f}(\mO_\mB,\mO_\mB)\Rightarrow\Win{II}{g}(\mO_\mB,\mO_\mB).\]
Indeed, if $\mu$ is a strategy for $\pl{II}$ in $\game{f}$, we fix a sequence $(n_i)_{i\in\omega}$ of natural numbers such that $g(n_i)\geq f(i)$ for all $i<\omega$ and then we induce a winning strategy for $\pl{II}$ in the game $\game{g}$ by using $\mu$ only in the innings $n\in\{n_i:i<\omega\}$.
\end{proof}

Now, we can translate Corollary~\ref{easy} for function spaces automatically.

\begin{corollary}
Let $X$ be a Tychonoff space and let $\mathcal{B}$ be a bornology on $X$ with a compact base. For a function $f\colon \omega\to[2,\aleph_0)$, the following are equivalent:
\begin{enumerate}
\item $\los{I}{1}(\Omega_{\0},\Omega_{\0})$;
\item $\Los{I}{f}(\Omega_{\0},\Omega_{\0})$;
\item $\sel{f}(\Omega_{\0},\Omega_{\0})$;
\item $\sone(\Omega_{\0},\Omega_{\0})$.
\end{enumerate}
\end{corollary}
\begin{remark}
We still do not know if the games $\game{f}(\mO,\mO)$ and $\gone(\mO,\mO)$ are equivalent for arbitrary topological spaces. If it become to be true, at least for T$_0$-spaces, then an analogous result can be derived for function spaces with the tools we presented above.
\end{remark}

\section{Bornologies and filters}\label{filters}

Recall that a filter $\mF$ on a set $C$ is a family of subsets of $C$ closed upwards and closed for taking finite intersections -- it is called a proper filter if $\emptyset\not\in \mF$. For a topological space $Y$ and a point $y\in Y$, we consider the neighborhood filter of $y$, 
\begin{equation}\mN_{y,Y}:=\{N\subset Y:\exists V\subset Y\text{, }V\text{ is open and }y\in V\subset N\}.\end{equation} In this section we intend to generalize Theorem 3 in \cite{jordan2007}, but first we make the necessary definitions, adapted from \cite{jordan2006,jordan2007}. 

For a fixed set $C$, we denote by $\mathbb{F}(C)$ the family of the proper filters of $C$, and for a cardinal $\kappa\geq \aleph_0$ we let $\mathbb{F}_\kappa(C)$ be the family of those proper filters of $C$ of the form
\begin{equation}\label{filtrogerado}\mG^\uparrow:=\{F\subset Y:\exists G\in\mG\left(G\subset F\right)\},\end{equation}
 where $\mathcal{G}\in[\wp(C)]^{\leq \kappa}$ -- particularly, we call the elements in $\mathbb{F}_1(C)$ and $\mathbb{F}_{\aleph_0}(C)$ as principal filters and countable based filters, respectively. If $R\subset C\times D$ is a binary relation on the sets $C$ and $D$ and if $\mathcal{F}$ is a collection of subsets of $C$, we set
$R\left(\mF\right):=\{R[F]:F\in\mF\}^{\uparrow}$. It is worth to mention that the correspondence 
\[\wp\left(\wp(C)\right)\ni \mG\longmapsto \mG^{\uparrow}\in\wp\left(\wp(C)\right),\]
does not determine a function $\wp\left(\wp(C)\right)\to \mathbb{F}(C)$. In fact, $\mG^{\uparrow}$ is a proper filter if and only if $\emptyset\not\in\mG$ and for all $A,B\in\mG$ there is a $G\in\mG$ such that $G\subset A\cap B$.

By a class of filters $\mathbb{K}$ we mean a property about filters, and we write $\mF\in\mathbb{K}$ to indicate that the filter $\mF$ has property $\mathbb{K}$. We also say that a topological space $Y$ is a $\mathbb{K}$-space if $\mN_{y,Y}\in\mathbb{K}$ for all $y\in Y$. For instance, the class of $\mathbb{F}_{\aleph_0}$-spaces is the class of spaces with countable character.

We say that a class of filters $\KK$ is $\FF_1$-composable if for any sets $C,D$ and any relation $R\subset C\times D$, the following holds: 
\begin{equation}\label{composable}\mF\in\mathbb{F}(C)\cap \mathbb{K}\text{ and }R(\mF)\in\mathbb{F}(D)\Rightarrow R(\mF)\in\mathbb{K}.\end{equation}

Note that the condition ``$R(\mF)\in\mathbb{F}(D)$'' in the left hand of the above implication is necessary, because in general there is no guarantee that $R(\mF)$ is a proper filter on $D$. If $R$ is just a relation in $C\times D$, it may happens that $R[F]=\emptyset$ for some $F\in\mF$, and in this case $R(\mF)$ is not a proper filter. However, the situation becomes simpler if $R$ is a function.
\begin{lemma}\label{aha}
Let $C$ and $D$ be sets and let $f\colon C\to D$ be a function.
\begin{enumerate}
\item\label{aha1} If $\mF\in\mathbb{F}(C)$, then $f(\mF)\in\mathbb{F}(D)$.
\item\label{aha2} If $\mG\in\mathbb{F}(D)$, then $f^{-1}(\mG)\in\FF(C)$ if and only if $G\cap f[C]\ne\emptyset$ for all $G\in\mG$.
\end{enumerate}
\end{lemma}

Finally, a class of filters $\KK$ is called $\FF_{\omega}$-steady if for any set $C$ and each pair of filters $\mF\in\FF(C)\cap\KK$ and $\mG\in\FF_{\omega}(C)$ such that $F\cap G\ne\emptyset$ for all $(F,G)\in\mF\times\mG$, the following holds
\begin{equation}\label{steady}\mF\vee\mG:=\{F\cap G:(F,G)\in\mF\times\mG\}^{\uparrow}\in\KK.\end{equation}

Given a Tychonoff space $(X,\tau)$ and a bornology $\mB$ on $X$, we denote by $\Gamma_{\mB}(X)$ the filter on $\tau$ generated by the sets \begin{equation}V(B):=\{U\in\tau:{B}\subset U\},\end{equation} i.e., $\Gamma_{\mB}(X):=\{V(B):B\in\mB\}^\uparrow$. Note that whenever $\mB'\subset \mB$ is a base for $\mB$, then $\{V(B):B\in\mB'\}^\uparrow =\Gamma_{\mB}(X)$.

In \cite{jordan2007}, Jordan considers the filter $\Gamma(X)$ on $\tau$, which coincides with $\Gamma_{[X]^{<\omega}}(X)$ according to our previous definition. Jordan shows that if a class of filters $\KK$ is $\FF_1$-composable and $\FF_\omega$-steady, then $\Gamma(X)\in \KK$ if and only if $C_{p}(X)$ is a $\KK$-space. Our next results aim to extend this conclusion to $\Gamma_{\mathcal{B}}(X)$ and $C_{\mathcal{B}}(X)$. Their proofs are natural adaptations from \cite{jordan2007}.

\begin{proposition}\label{halfultimate}
Let $(X,\tau)$ be a Tychonoff space and let $\mB$ be a bornology with a compact base on $X$. Suppose that $\KK$ is an $\FF_1$-composable class of filters. If $C_{\mB}(X)$ is a $\KK$-space, then $\Gamma_{\mB}(X)\in\KK$.
\end{proposition}
\begin{proof} It is enough to present a neighborhood filter $\mF$ in $C_{\mB}(X)$, a set $Y$ with functions $\pi\colon Y\to\tau$ and $\Phi\colon Y\to C_{\mB}(X)$ such that $\Phi^{-1}(\mF)\in\FF(Y)$ and $\Gamma_\mB(X)=\pi(\Phi^{-1}(\mF))$. In fact, since $\mF\in\KK$ and $\KK$ is $\FF_1$-composable, it follows by the previous lemma that $\Phi^{-1}(\mF)\in\KK$ and $\pi(\Phi^{-1}(\mF))\in\KK$.

Let $\mB_0$ be a compact base for $\mB$. Let $Y:=\{(B,U)\in\mB_0\times\tau:B\subset U\}$ and define $\pi:Y\rightarrow \tau$ by $\pi(B,U)=U$. Since $X$ is a Tychonoff space, for each $(B,U)\in Y$ there exists some $f=f_{(B,U)}\in C_{\mB}(X)$ such that $f\upharpoonright B\equiv 0$ and $f\upharpoonright X\setminus U\equiv 1$, so we may set $\Phi:Y\rightarrow C_{\mB}(X)$ by $\Phi(B,U)=f_{(B,U)}$ for each $(B,U)\in Y$.

Now, we take $\mF:=\mN_{\0,C_{\mB}(X)}=\{\langle B,\frac{1}{n+1}\rangle[\0]:B\in\mB_0,n\in\omega\}^\uparrow$. In order to prove $\pi(\Phi^{-1}(\mF))=\Gamma_{\mB}(X)$, it is enough to note that both filters are generated by the same base. Indeed, for any $B\in\mB_0$ and $\varepsilon\in(0,1)$ one has \[V(B)=\pi\left[\Phi^{-1}\left[\left\langle B,\varepsilon\right\rangle [\0]\right]\right],\] from which the desired equality follows.
\end{proof}

\begin{proposition}\label{halfultimate2}
Let $\KK$ be a class of filters that is $\FF_1$-composable and $\FF_\omega$-steady. For a topological space $(X,\tau)$ and a bornology $\mB$ on $X$, $\Gamma_{\mB}(X)\in\KK$ implies that $C_{\mB}(X)$ is a $\KK$-space.
\end{proposition}

\begin{proof}
We use a similar strategy as in the proof of the previous proposition. We define a set $Y$ with a proper countably based filter $\mathcal{H}\in\FF_\omega(Y)$ and functions $\pi\colon Y\to C_\mB(X)$ and $\Phi\colon Y\to \tau$ such that $\Phi^{-1}(\Gamma_\mB(X))$ and $\Phi^{-1}(\Gamma_\mB(X))\vee \mathcal{H}$ are proper filters in $Y$, and $\pi(\Phi^{-1}(\Gamma_\mB(X))\vee \mathcal{H})=\mN_{\0,\mC_\mB(X)}$. Again, the hypotheses over $\KK$ guarantee that $\pi(\Phi^{-1}(\Gamma_\mB(X))\vee\mathcal{H})\in\KK$, which is enough to finish the proof, since $C_{\mB}(X)$ is a homogeneous space.

For brevity, we call $I_n:=(-\frac{1}{n+1},\frac{1}{n+1})\subset\RR$ for each $n\in\omega$.

Let $Y:=\{(f,B,n)\in C_{\mB}(X)\times \mB\times\omega:f[{B}]\subset I_n\}$ and $\mathcal{H}:=\{M_n:n\in\omega\}^{\uparrow}$, where $M_n:=\{(f,B,m)\in Y:m\geq n\}$ for each $n\in\omega$. Now, let $\pi:Y\rightarrow C_{\mB}(X)$ be defined by $\pi(f,B,n)=f$ and $\Phi:Y\rightarrow \tau$ defined as $\Phi(f,B,n)=f^{-1}[I_n]$. Since $\Gamma_{\mB}(X)\in\KK$ and $\KK$ is $\FF_1$-composable, it follows that $\Phi^{-1}(\Gamma_{\mB}(X))\in\KK$. 
Also, since $\Phi^{-1}[V(B)]\cap M_n\ne\emptyset$ for all $n\in\omega$, the $\FF_\omega$-steadiness of $\KK$ gives $\Phi^{-1}(\Gamma_{\mB}(X))\vee \mathcal{H}\in\KK$ and, again by the $\FF_1$-composability of $\KK$, $\pi(\Phi^{-1}(\Gamma_{\mB}(X))\vee\mathcal{H})\in\KK$. So, in order to finish the proof we must show that $\mN_{\0,C_\mB(X)}=\pi(\Phi^{-1}(\Gamma_{\mB}(X))\vee\mathcal{H})$. The desired equality follows because
\[\left\langle B,\frac{1}{n+1}\right\rangle[\0]=\pi\left[\Phi^{-1}\left[V(B)\right]\cap M_n\right]\]
holds for any $B\in\mB$ and $n\in\omega$.
\end{proof}

Altogether, the propositions above yields the following

\begin{theorem}\label{ultimate}
Let $X$ be a Tychonoff space and let $\mB$ be a bornology with a compact base on $X$. If $\KK$ is a class of filters that is $\FF_1$-composable and $\FF_\omega$-steady, then $\Gamma_{\mB}(X)\in\KK$ if and only if $C_{\mB}(X)$ is a $\KK$-space.
\end{theorem}

\begin{example}
For a cardinal $\kappa\geq \aleph_0$, we say that a filter $\mF$ on $C$ belongs to $\mathbb{T}_\kappa$ if for all $A\subset C$ such that $A\cap F\ne \emptyset$ holds for every element of $\mF$, there is a $B\in[A]^{\leq \kappa}$ such that $B\cap F\ne\emptyset$ for all $F\in\mF$. It can be shown that $\mathbb{T}_\kappa$ is a class of filters $\FF_1$-composable and $\FF_\omega$-steady. So, under the assumptions of the previous corollary, $C_\mB(X)$ has tightness less than or equal to $\kappa$ if and only if $\Gamma_\mB(X)\in\mathbb{T}_\kappa$, i.e., any $\mB$-covering of $X$ has a $\mB$-subcovering of cardinality $\leq\kappa$, a result originally due to McCoy and Ntantu~\cite{McCoybook}.
\end{example}

\begin{example}
Let $\mF$ be a proper filter on a set $C$, and consider the family 
\[\mathcal{M}:=\{A\subset C:\forall F\in\mF(A\cap F\ne\emptyset)\}.\]

Note that for any function $\varphi\colon\omega\to[2,\aleph_0]$, both classes of filters $\sel{\varphi}(\mathcal{M},\mathcal{M})$ and ${\Win{II}{\varphi}(\mathcal{M},\mathcal{M})}$ are $\FF_1$-composable. Thus, the directions ``property in $C_\mB(X)$ implies property in $X$'' of Proposition~\ref{prop1} concerning both $\sel{\varphi}$ and $\pl{II}$ follow from Proposition~\ref{halfultimate}. On the other hand, the converses follow from Proposition~\ref{halfultimate2} whenever $\varphi$ is a constant function, because in this case it can be proved that $\sel{\varphi}(\mathcal{M,M})$ and $\Win{II}{\varphi}(\mathcal{M,M})$ are also $\FF_{\omega}$-steady classes of filters.
However, we were not able to prove similar statements regarding $\pl{I}$ in the game $\game{\varphi}(\mathcal{M,M})$.
\end{example}	

\section{$\gamma$-productive spaces}\label{productive}

Gerlits and Nagy \cite{Gerlits1982} had introduced the concept of point-cofinite open coverings\footnote{Usually called $\gamma$-coverings in the literature.} in their analysis of Fr\'echet property in $C_p(X)$. 
Recall that a topological space $Y$ is Fr\'echet (resp. strictly Fr\'echet) if for each $y\in Y$ and for each $A\in\Omega_y$ there is a subset $B\subset A$ such that $B\in\Gamma_y$, where \[\Gamma_y:=\{A\subset Y:|A\setminus V|<\aleph_0\text{ for all }V\in\mN_{y,Y}\},\] (resp. if $\sone(\Omega_y,\Gamma_y)$ holds). We say that an infinite collection $\mU$ of proper open sets of $X$ is called a point-cofinite covering if for all $x\in X$ the set $\{U\in\mU:x\not\in\mU\}$ is finite, and we denote by $\Gamma$ the collection of all point-cofinite coverings of $X$ - particularly, note that $\Gamma\subset \Omega$. A space $X$ is called a $\gamma$-space if any nontrivial $\omega$-covering has a (countable) $\gamma$-subcovering. Then we have the following

\begin{theorem}
[Gerlits and Nagy\cite{Gerlits1982}]\label{frechetG} For a Tychonoff space $X$, the following are equivalent:
\begin{enumerate}
\item $X$ is a $\gamma$-space;
\item $\sone(\Omega,\Gamma)$ holds;
\item $C_p(X)$ is a strictly Fr\'echet space;
\item $C_p(X)$ is a Fr\'echet space.
\end{enumerate}
 
\end{theorem}

In \cite{jordan2007}, Jordan obtain a variation of the above theorem as a corollary of Theorem \ref{ultimate}, with appropriate definitions for Fr\'echet filters and \emph{strongly}\footnote{Strictly Fr\'echet spaces and strongly Fr\'echet spaces are not formally the same: the later were independently introduced by Michael~\cite{Michaelfrechet} and Siwiec~\cite{Siwiec}. Although the definition is similar, in the strong case there is the additional requirement for the sequence $(A_n)_{n\in\omega}$ with $y\in\bigcap_{n\in\omega}\overline{A_n}$ to be decreasing.} Fr\'echet filters, which turns out to be $\FF_1$-composable and $\FF_\omega$-steady classes of filters.

However, we shall pay more attention to the following characterization. 
\begin{theorem}[Jordan and Mynard\cite{jordan2004}]\label{prodfrechet} A topological space $Y$ is productively Fr\'echet if and only if $Y\times Z$ is a Fr\'echet space for any strongly Fr\'echet space $Z$.
\end{theorem}

In the above theorem, the sentence ``$Y$ is productively Fr\'echet'' has a very precise meaning -- it is a space such that all of its neighborbood filters belongs to the class of productively Fr\'echet filters:
\begin{center}
 $\mF\in\FF(C)$ is productively Fr\'echet if for each strongly Fr\'echet filter $\mathcal{H}\in\FF(C)$ such that $F\cap H\ne\emptyset$ for all $(F,H)\in\mF\times\mathcal{H}$ there is a filter $\mG\in\FF_{\omega}(C)$ refining $\mF\vee \mathcal{H}$.
\end{center}

The considerations above justify the following definition of Jordan \cite{jordan2007}: a Tychonoff space $X$ is called $\gamma$-productive if the filter $\Gamma(X)$ is productively Fr\'echet. Then, by using the fact that productively Fr\'echet filters are $\FF_1$-composable and $\FF_\omega$-steady, Jordan proves the following.
\begin{proposition}[Jordan \cite{jordan2007}]\label{corjordan}
If $X$ is $\gamma$-productive, then $X\times Y$ is a $\gamma$-space for all $\gamma$-space $Y$.
\end{proposition}

By extending this definition to the filter $\Gamma_{\mathcal{B}}(X)$, we will prove that the \emph{productivity} of $\gamma$-productive spaces is far more strong. In order to do this, we will follow the indirect approach presented by Miller, Tsaban and Zdomskyy in \cite{tsaban2016} to derive Proposition~\ref{corjordan}.

Let $X$ be a Tychonoff space and let $\mathcal{B}$ be a bornology  on $X$. We say that an infinite collection $\mU$ of proper open sets of $X$ is a $\mB$-cofinite covering if for all $B\in\mB$ the set $\{U\in\mU:B\not\subset U\}$ is finite, and we denote by $\Gamma_{\mB}$ the collection of all $\mB$-cofinite coverings of $X$. Naturally, we say that $X$ is a $\gamma_\mB$-space if any nontrivial $\mB$-covering has a (countable) $\mB$-cofinite subcovering.

McCoy and Ntantu~\cite{McCoybook} have introduced $\mB$-cofinite open coverings in their generalization of Theorem~\ref{frechetG}, calling them as \emph{$\mB$-sequences} there. Although they just stated items (\ref{enrolo1}) and (\ref{enrolo4}) of the theorem below, their arguments, which are adapted from Gerlits and Nagy~\cite{Gerlits1982}, can be used to prove the following.

\begin{theorem}[McCoy and Ntantu~\cite{McCoybook}]
Let $X$ be a Tychonoff space and let $\mB$ be a bornology with a compact base on $X$. The following are equivalent:
\begin{enumerate}
\item\label{enrolo1} $X$ is a $\gamma_\mB$-space;
\item $\sone(\mathcal{O_B},\Gamma_\mathcal{B})$ holds;
\item $C_{\mB}(X)$ is strictly Fr\'echet;
\item\label{enrolo4} $C_{\mB}(X)$ is Fr\'echet.
\end{enumerate}
\end{theorem} 

\begin{remark}
In particular, all of the above conditions are equivalent to require that $C_{\mB}(X)$ is strongly Fr\'echet.
\end{remark}

Now, we will say that a Tychonoff space $X$ with a bornology $\mathcal{B}$ is $\gamma_\mB$-productive if the filter $\Gamma_{\mB}(X)$ is productively Fr\'echet -- note that for $\mB=[X]^{<\aleph_0}$, one obtains the original definition of $\gamma$-productive spaces.

Since we will work with the product of spaces endowed with different bornologies, we need to describe their behavior under products.

\begin{proposition} Given a family $\{X_t:t\in T\}$ of pairwise disjoint topological spaces, consider for each $t\in T$ a set $\mB_t\subset\wp(X_t)$. Let $\mB_0=\{\prod_{t\in T}B_t:\forall t(B_t\in \mB_t)\}$ and $\mB_1=\{\bigsqcup_{t\in T}B_t:B_t\in\mB_t$ for finitely many $t\in T$, $B_t=\emptyset$ otherwise$\}$. If $\mB_t$ is a (compact) base for each $t\in T$, then $\mB_0$ and $\mB_1$ are (compact) bases for bornologies on $\prod_{t\in T}X_t$ and $\sum_{t\in T}X_t$, respectively. 
\end{proposition}

We denote by $\bigotimes_{t\in T}\mB_t$ and $\bigoplus_{t\in T}\mB_t$ the bornologies generated by the bases $\mB_0$ and $\mB_1$ in the above proposition, respectively.

\begin{proposition}
Let $\{X_t:t\in T\}$ be a family of topological spaces, and for each $t\in T$ let $\mB_t$ be a bornology on $X_t$. Then $C_{\bigoplus_{t\in T}\mB_t}(\sum_{t\in T}X_t)$ is homeomorphic to $\prod_{t\in T}C_{\mB_t}(X_t)$.
\end{proposition}
\begin{proof}
Note that the map
\[C_{\bigoplus_{t\in T}\mB_t}\left(\sum_{t\in T}X_t\right)\ni f\longmapsto (f\upharpoonright X_t)_{t\in T}\in\prod_{t\in T}C_{\mB_t}(X_t)\]
is continuous and it has a continuous inverse.
\end{proof}

\begin{remark}Particularly, it follows from the previous proposition that 
\begin{equation}
C_p\left(\sum_{t\in T}X_t\right)\textnormal{is homeomorphic to}\prod_{t\in T}C_p(X_t),
\end{equation}
and, if each $X_t$ is a Hausdorff space, then 
\begin{equation}
C_k\left(\sum_{t\in T}X_t\right)\textnormal{is homeomorphic to}\prod_{t\in T}C_k(X_t).
\end{equation}

For brevity, if $X_t=X$ and $\mB_t=\mB$ for all $t\in T$, we will write $\mB^{|T|}$ instead of $\bigotimes_{t\in T}\mB$.\end{remark}

The next lemma will be very useful later.
\begin{lemma}
[Miller, Tsaban and Zdomskyy~\cite{tsaban2016}]\label{trick} Let $\mathfrak{P}$ be a topological property hereditary for closed subspaces and preserved under finite power. Then for any pair of topological spaces $X$ and $Y$, $X\times Y$ has the property $\mathfrak{P}$ provided that $X+ Y$ has the property $\mathfrak{P}$.
\end{lemma}

The first three items in the next proposition states that the property ``having a bornology $\mB$ with a compact base such that it is a $\gamma_\mB$-space'' satisfies the conditions of the previous lemma.
\begin{proposition}
Let $X$ be a topological space and let $\mB$ be a bornology on $X$.
\begin{enumerate}[(a)]
\item If $Y\subset X$, then $\mB_Y:=\{B\cap Y:B\in \mB\}$ is a bornology on $Y$. If $Y$ is closed and $\mB$ has a compact base on $X$, then $\mB_Y$ has a compact base on $Y$.
\item If $X$ is a $\gamma_\mB$-space and $Y\subset X$ is closed, then $Y$ is a $\gamma_{\mB_Y}$-space.
\item If $X$ is a $\gamma_\mB$-space and $\mB$ has a compact base, then $X^n$ is a $\gamma_{\mB^n}$-space for any $n\in\omega$.
\item\label{voltagama} If $X$ is a $\gamma_\mB$-space and $Y$ is a $\gamma_{\mL}$-space for a bornology $\mL$ in $Y$ such that $X\times Y$ is a $\gamma_{\mB\otimes\mL}$-space, then $X\sqcup Y$ is a $\gamma_{\mB\oplus\mL}$-space.
\end{enumerate}
\end{proposition}

\begin{corollary}
Let $X$ and $Y$ be topological spaces with bornologies $\mB$ and $\mL$, respectively, both of them with compact bases. Then $X\times Y$ is a $\gamma_{\mB\otimes\mL}$-space if and only if $X+ Y$ is a $\gamma_{\mB\oplus \mL}$-space.
\end{corollary}

We finally can state and prove the desired extension of Proposition \ref{corjordan}.

\begin{corollary}
Let $X$ be a Tychonoff space and let $\mB$ be a bornology with a compact base on $X$. If $X$ is $\gamma_\mB$-productive, then $X\times Y$ is a $\gamma_{\mB\otimes \mL}$-space for any Tychonoff space $Y$ endowed with a bornology $\mL$ with a compact base such that $Y$ is a $\gamma_{\mL}$-space.
\end{corollary}
\begin{proof}
If $X$ is $\gamma_\mB$-productive then $C_\mB(X)$ is productively Fr\'echet. Since $Y$ is a $\gamma_{\mL}$-space, it follows that $C_{\mL}(Y)$ is Fr\'echet, hence $C_\mB(X)\times C_\mL(Y)$ is (strongly) Fr\'echet. But $C_\mB(X)\times C_\mL(Y)$ is homeomorphic to $C_{\mB\oplus \mL}(X+ Y)$, thus $X+ Y$ is a $\gamma_{\mB\oplus \mL}$-space, and the conclusion follows from the last corollary.
\end{proof}

Particularly, if a Tychonoff space $X$ is $\gamma$-productive, then $X\times Y$ is a $\gamma_{[X]^{<\omega}\otimes \mL}$-space whenever $Y$ is a Tychonoff $\gamma_{\mL}$-space, where $\mL$ is a bornology with a compact base on $Y$. This suggests that the converse of Proposition \ref{corjordan} may be false.

\bibliography{references}{}
\bibliographystyle{abbrv}
\end{document}